\documentclass[12pt,draft]{amsart}

\usepackage{color,enumerate}

\newtheorem{theorem}{Theorem}[section]

\theoremstyle{definition}

\theoremstyle{remark}

\numberwithin{equation}{section}

\usepackage[margin=2.4cm]{geometry}

\newcommand{\R}{\mathbb{R}}

\newcommand{\loc}{\textrm{loc}}
\newcommand{\oute}{\textrm{out}}
\newcommand{\inn}{\textrm{in}}

\def\polhk#1{\setbox0=\hbox{#1}{\ooalign{\hidewidth
    \lower1.5ex\hbox{`}\hidewidth\crcr\unhbox0}}}

\begin{document}

\title[Gradient singularity]
{Moving gradient singularity\\ for the evolutionary $p$-Laplace equation}

\author[E. Lindgren]{Erik Lindgren}
\address{Department of Mathematics, KTH - Royal Institute of Technology, SE-100 44 Stockholm, Sweden. }
\email{eriklin@kth.se}

\author[J. Takahashi]{Jin Takahashi}
\address{Department of Mathematical and Computing Science, 
Tokyo Institute of Technology, 2-12-1 Ookayama, Meguro-ku, Tokyo 152-8552, Japan}
\email{takahashi@c.titech.ac.jp}

\dedicatory{Dedicated to the memory of Professor Marek Fila}



\begin{abstract}
We consider the evolutionary $p$-Laplace equation in $\R^n$. 
For $p>n$, we construct a solution $u$ with 
a moving gradient singularity in the sense that 
$|\nabla u(x,t)|\to \infty$ for each $t$ as $x\to\xi(t)$, 
where $\xi:[0,\infty)\to\R^n$ is a given curve. 
\end{abstract}

\maketitle

\section{Introduction}
This paper is concerned with the evolutionary $p$-Laplace equation 
\begin{equation}
\label{eq:evop}
	u_t = \Delta_p u 
\end{equation}
in $\R^n$, where $n\geq1$, $p>1$ 
and $\Delta_p u := \nabla \cdot (|\nabla u|^{p-2} \nabla u)$ is the $p$-Laplacian.

The $p$-Laplace equation, $\Delta_p u =0$, arises as the Euler-Lagrange equation of the functional
\[
u\mapsto \int  |\nabla u|^p\,dx.
\] 
See \cite{Lin06} for an introduction and an overview of the theory for the $p$-Laplace equation.

Equation \eqref{eq:evop}, the evolutionary $p$-Laplace equation, is the corresponding gradient flow. For its regularity theory, see for instance \cite{DiB}.
The fundamental solution for the $p$-Laplace equation is
$(n-p)|x|^{-(n-p)/(p-1)}$ for $p\neq n$. 
In the case $p>n$, this can be regarded as 
a solution of \eqref{eq:evop} with a standing gradient singularity at $x=0$ for each $t$. 
Then, a simple question arises as to whether 
there exists a solution with a moving gradient singularity. 
Here, by a solution with a moving gradient singularity, 
we mean a locally bounded solution $u$ satisfying 
\[
	|\nabla u(x,t)| \to \infty 
	\quad \mbox{ for each }t \mbox{ as }x\to \xi(t), 
\]
where $\xi:[0,\infty)\to \R^n$ is a given curve. 
In this paper, we give an example of 
such a singular solution for the evolutionary $p$-Laplace equation with $p>n$.

Solutions with a moving singularity, 
in the sense that 
$|u(x,t)| \to \infty$ for each $t$ as $x\to \xi(t)$,  
were studied for various kinds of nonlinear parabolic equations, 
see \cite{KT17,SY09,Ta21} for the semilinear heat equation, 
\cite{FMTY22,FTY19} for the porous medium and the fast diffusion equation, 
\cite{KZ15,KUW24} for the Navier-Stokes equations and the references given there.  It is also worth mentioning \cite{KL05} and \cite{KLP16}, where supersolutions of \eqref{eq:evop} are studied in detail. However, to our best knowledge, 
solutions with a moving gradient singularity have not been constructed yet.

To construct a solution with a moving gradient singularity, 
we consider the following initial value problem: 
\begin{equation} \label{eq:main}
	\left\{ 
	\begin{aligned}
		&u_t=\Delta_p u,  
		&&x \in \R^n \setminus \{ \xi(t) \}, \;  t>0, \\
		&u(x,0) = u_0(x),  
		&&x\in  \R^n\setminus \{ \xi(0)\}.  \\
	\end{aligned}
	\right.
\end{equation}
Our main result is as follows.

\begin{theorem}\label{th:main}
Let $n\geq2$, $p>n$ and $k>k'>0$. Fix $\lambda$ and $\lambda'$ such that 
\begin{equation}\label{eq:lamcon}
	0<\lambda < \frac{p-n}{p-1} < \lambda'<1. 
\end{equation}
Then, there exist $0<C_\xi<1$ and $A>1$ such that 
the following statement holds: 
Assume that $\xi \in C^1([0,\infty); \R^n)$ and 
$u_0\in C(\R^n) \cap C^1(\R^n\setminus\{\xi(0)\})$ satisfy 
\begin{align}
	&|\xi'(t)| \leq C_\xi (1+At)^{-1} \quad \mbox{ for any }t\geq0,  \label{eq:assum} \\
	&\begin{aligned}
	&k' |x-\xi(0)|^{\lambda'}  (1+|x-\xi(0)|)^{\lambda-\lambda'} \\
	& \leq u_0(x) \leq k |x-\xi(0)|^\lambda (1+|x-\xi(0)|)^{\lambda'-\lambda} 
	\quad \mbox{ for any }x\in\R^n,  \end{aligned} \label{eq:u0as} 
\end{align}
respectively. 
Then, \eqref{eq:main} admits a nonnegative 
weak solution $u$ satisfying (i), (ii) and (iii).  
\begin{enumerate}[(i)]
\item
$u\in C(\R^n\times[0,\infty)) \cap 
C^1(\{ (x,t); x\in \R^n\setminus\{\xi(t)\}, t\in (0,\infty) \} )$. 
\item
For each $0<T<\infty$, there exist $0<r_T<1$  and $C_T>1$ 
depending only on $n$, $p$, $\lambda$, $\lambda'$, $A$, $C_\xi$ and $T$ such that 
\[
	C_T^{-1} |x-\xi(t)|^{\lambda'} \leq u(x,t) \leq C_T |x-\xi(t)|^\lambda, 
	\quad 0\leq |x-\xi(t)| \leq r_T, \; 0\leq t\leq  T. 
\]
\item
For each $t\in[0,\infty)$, 
\[
	|\nabla u(x,t)|\to \infty \mbox{ as }x\to\xi(t). 
\]
\end{enumerate}
\end{theorem}

The rest of this paper is devoted to the proof of Theorem \ref{th:main} 
and is organized as follows. 
In Section \ref{sec:comp}, we construct suitable 
super- and subsolutions. 
In Section \ref{sec:comp}, we show the existence of a 
solution with a moving gradient singularity 
based on a monotone iteration scheme.

\section{Comparison functions}\label{sec:comp}
Let $n\geq2$ and $p>n$. 
Set $v(y,t):=u(y+\xi(t),t)$. 
We consider 
\begin{equation}\label{eq:prad}
\left\{ 
\begin{aligned}
	&\partial_t v = \Delta_p v  + \xi'(t) \cdot \nabla v,  
	&&y\in \R^n\setminus \{0\}, \; t>0, \\
	&v_0(y) = u_0(y+\xi(0)),  
	&&y\in  \R^n\setminus \{0\}, 
\end{aligned}
\right.
\end{equation}
where $u_0$ satisfies \eqref{eq:u0as} with 
$k$, $k'$, $\lambda$ and $\lambda'$ chosen as in Theorem \ref{th:main}. 
In this section, we construct radial super- and subsolutions 
$v^+$ and $v^-$ of \eqref{eq:prad} satisfying 
\begin{equation}\label{eq:order}
	v^-(|y|,t) \leq 
	k' |y|^{\lambda'}  (1+|y|)^{\lambda-\lambda'}
	\leq 
	k |y|^\lambda (1+|y|)^{\lambda'-\lambda}
	\leq v^+(|y|,t) 
\end{equation}
for $y\in\R^n$ and $t \geq 0$. 
Roughly speaking, $v^+$ and $v^-$ behave like 
\[
	v^-(|y|,t) = f(t) |y|^{\lambda'},\quad 
	v^+(|y|,t) = g(t) |y|^\lambda 
	\quad 
	\mbox{ near }y=0 \mbox{ for any }t\geq 0
\]
with some functions $f$ and $g$.

\subsection{Supersolution}
We construct a supersolution. 
Let $\lambda$ and $\lambda'$ satisfy \eqref{eq:lamcon}. 
Let $K>k$ be a large constant specified later. 
We write $r:=|y|$. Set 
\[
\left\{ 
\begin{aligned}
	&v_\inn^+(r,t):= K(1+At)^{\theta(\lambda'-\lambda)} r^\lambda, \\
	&v_\oute^+(r,t):= K(1+R)^{-\lambda'} \left( r+R(1+At)^\theta \right)^{\lambda'}, 
\end{aligned}
\right.
\]
where $A>1$ is given in \eqref{eq:assum} and 
\begin{equation}\label{eq:Rtdef}
	R:=\frac{\lambda'}{\lambda}>1, \quad 
	\theta:= \frac{1}{p-1-(p-2)\lambda'}>0. 
\end{equation}
We denote by $\rho(t):=(1+At)^\theta$ an intersection point 
of $v_\inn^+$ and $v_\oute^+$. 
Define 
\[
	v^+(y,t):=
	\left\{ 
	\begin{aligned}
	&v_\inn^+(|y|,t) && \mbox{ for }0\leq |y|\leq \rho(t), \; t\geq 0,\\
	&v_\oute^+(|y|,t) && \mbox{ for }|y|> \rho(t), \; t\geq0.  
	\end{aligned}
	\right.
\]
In what follows, we write $v^+(y,t)=v^+(r,t)$ by abuse of notation.  

By the following 4 steps, 
we check that $v^+$ becomes a supersolution of \eqref{eq:prad} satisfying 
\begin{equation}\label{eq:v+ord}
	v^+(r,t) \geq k r^\lambda (1+r)^{\lambda'-\lambda}
	\quad 
	\mbox{ for }r\geq0, \; t\geq 0. 
\end{equation}

Step 1 (Ordering condition). 
We check \eqref{eq:v+ord}. 
For $0\leq r \leq \rho(t)$ and $t\geq0$, since $\rho\geq(1+r)/2$ 
and $\lambda'>\lambda$, we have 
\[
	v^+(r,t)=v_\inn^+(r,t)=  K\rho^{\lambda'-\lambda} r^\lambda
	\geq 
	K\left( \frac{1+r}{2} \right)^{\lambda'-\lambda} r^\lambda
	\geq  k r^\lambda (1+r)^{\lambda'-\lambda} 
\] 
if $K(1/2)^{\lambda'-\lambda}\geq k$. 
As for $r > \rho(t)$, since $\lambda' \rho/\lambda \geq 1$, we see that 
\[
	v^+(r,t)=v_\oute^+(r,t)
	= K\left( 1+\frac{\lambda'}{\lambda} 
	\right)^{-\lambda'} \left( r+ \frac{\lambda'}{\lambda} \rho \right)^{\lambda'-\lambda} 
	\left( r+ \frac{\lambda'}{\lambda} \rho \right)^\lambda
	\geq  k ( r+ 1)^{\lambda'-\lambda} r^\lambda 
\]
provided that 
$K( 1+(\lambda'/\lambda))^{-\lambda'} \geq k$. 
Then we can choose $K$ such that \eqref{eq:v+ord} holds. 
We note that $K$ is determined by $\lambda$, $\lambda'$ and $k$. 
For later use, we take $K$ such that 
\begin{equation}\label{eq:Kcon}
	K>\max\left\{ 1, 
	2^{\lambda'-\lambda} k, \left( 1+\frac{\lambda'}{\lambda} \right)^{\lambda'} k
	\right\}. 
\end{equation}

Step 2 (Matching condition). 
We need to check the validity of the appropriate matching condition 
$\partial_r v_\inn^+(\rho(t),t) \geq \partial_r v_\oute^+(\rho(t),t)$. 
By the choice of $R$ in \eqref{eq:Rtdef}, we have 
\[
\begin{aligned}
	&\partial_r v_\inn^+(\rho(t),t) - \partial_r v_\oute^+(\rho(t),t) 
	= K(1+At)^{\theta(\lambda'-\lambda)} \lambda \rho^{\lambda-1} 
	- K(R+1)^{-\lambda'} \lambda' \left( \rho+R(1+At)^\theta \right)^{\lambda'-1} \\
	&= 
	K(1+At)^{\theta(\lambda'-1)} 
	\left( \lambda - \frac{\lambda'}{R+1}\right) 
	= 
	K(1+At)^{\theta(\lambda'-1)} 
	\frac{\lambda^2}{\lambda+\lambda'} \geq 0 
\end{aligned}
\]
for $t\in [0,\infty)$. 
Then the matching condition is satisfied.

Step 3 (Inner part). 
We show that $v_\inn^+$ is a supersolution for $r\leq \rho(t)$. 
Recall the following formula: 
\[
	\Delta_p f(r)= r^{1-n} (r^{n-1} |f'(r)|^{p-2} f'(r))', \quad r=|y|. 
\]
Then, 
\[
\begin{aligned}
	&\partial_t v_\inn^+ -\Delta_p v_\inn^+ -\xi'(t) \cdot \nabla v_\inn^+ \\
	&= K\theta (\lambda'-\lambda) A (1+At)^{\theta(\lambda'-\lambda)-1} r^\lambda \\ 
	&\quad 
	+K^{p-1}\lambda^{p-1} ((1-\lambda)(p-1)-(n-1)) (1+At)^{\theta(\lambda'-\lambda)(p-1)} 
	r^{-1-(1-\lambda)(p-1)} \\
	&\quad 
	-K\xi'(t) \cdot (y/|y|) 
	(1+At)^{\theta(\lambda'-\lambda)} \lambda r^{\lambda-1}. 
\end{aligned}
\]
By \eqref{eq:lamcon}, we have $(1-\lambda)(p-1)-(n-1)>0$. 
This together with 
$\lambda'>\lambda$, $K^{p-1}>K$ and \eqref{eq:assum} shows that 
there exists $C=C(n,p,\lambda,\lambda')>1$ such that 
\[
\begin{aligned}
	&\partial_t v_\inn^+ -\Delta_p v_\inn^+ -\xi'(t) \cdot \nabla v_\inn^+ \\
	&\geq 
	K\lambda^{p-1} ((1-\lambda)(p-1)-(n-1)) (1+At)^{\theta(\lambda'-\lambda)(p-1)} 
	r^{-(1-\lambda)(p-1)-1} \\
	&\quad 
	-K\xi'(t) \cdot (y/|y|) 
	(1+At)^{\theta(\lambda'-\lambda)} \lambda r^{\lambda-1} \\
	&\geq 
	K C^{-1} (1+At)^{\theta(\lambda'-\lambda)(p-1)} 
	r^{-(1-\lambda)(p-1)-1} 
	- K C C_\xi  (1+At)^{\theta(\lambda'-\lambda) -1}  r^{\lambda-1} \\
	&= 
	K (1+At)^{\theta(\lambda'-\lambda)(p-1)} r^{-(1-\lambda)(p-1)-1} 
	\left(  C^{-1} 
	- C C_\xi  (1+At)^{-1 - \theta(\lambda'-\lambda)(p-2)} 
	r^{\lambda + (1-\lambda)(p-1)} 
	\right). 
\end{aligned}
\]
For $r\leq \rho(t)=(1+At)^\theta$, the choice of $\theta$ in \eqref{eq:Rtdef} gives 
\[
\begin{aligned}
	&C^{-1} - C C_\xi  (1+At)^{-1 - \theta(\lambda'-\lambda)(p-2)} 
	r^{\lambda + (1-\lambda)(p-1)} \\
	&\geq 
	C^{-1} - C C_\xi  (1+At)^{-1 - \theta(\lambda'-\lambda)(p-2)
	+ \theta (\lambda + (1-\lambda)(p-1))} \\
	&= 
	C^{-1} - C C_\xi  (1+At)^{-1 + \theta( p-1-(p-2)\lambda' )}
	= C^{-1} - C C_\xi. 
\end{aligned}
\]
Hence there exists $0<C_\xi<1$ depending only on $n$, $p$, $\lambda$, and $\lambda'$ 
such that $v_\inn^+$ is a supersolution for $r\leq \rho(t)$ and $t\geq 0$. 

Step 4 (Outer part). 
It remains to prove that $v_\oute^+$ is a supersolution for $r> \rho(t)$. 
Direct computations show that 
\[
\begin{aligned}
	&\partial_t v_\oute^+ -\Delta_p v_\oute^+ -\xi'(t) \cdot \nabla v_\oute^+ \\
	&= K A \lambda' R \theta (R+1)^{-\lambda'} (1+At)^{\theta-1} 
	\left( r+R(1+At)^{\theta} \right)^{\lambda'-1} \\
	&\quad 
	+ K^{p-1} (\lambda')^{p-1} (R+1)^{-\lambda'(p-1)} 
	\left( r+R(1+At)^\theta \right)^{(\lambda'-1)(p-1)-1} \\
	&\qquad \times 
	\left[ (1-\lambda')(p-1)  - (n-1)r^{-1} (r+R(1+At)^\theta) \right] \\
	&\quad 
	-K \lambda' \xi'(t) \cdot (y/|y|) (R+1)^{-\lambda'}  
	\left( r+R(1+At)^\theta \right)^{\lambda'-1}. 
\end{aligned}
\]
Then by \eqref{eq:assum}, $1+At\geq 1$ and $r> \rho(t)=(1+At)^\theta$, 
there exists $C=C(n,p,\lambda',R,\theta)>1$ such that 
\[
\begin{aligned}
	&\partial_t v_\oute^+ -\Delta_p v_\oute^+ -\xi'(t) \cdot \nabla v_\oute^+ \\
	&\geq  
	K (1+At)^{\theta-1} 
	\left( r+R(1+At)^{\theta} \right)^{\lambda'-1}  
	\left[ 
	C^{-1} A - C C_\xi (1+At)^{-\theta}
	\right] \\
	&\quad 
	- K^{p-1} C  
	\left( r+R(1+At)^\theta \right)^{(\lambda'-1)(p-1)-1} 
	\left[ 1 + \frac{R(1+At)^\theta}{r} \right] \\
	&\geq  
	K (1+At)^{\theta-1} 
	\left( r+R(1+At)^{\theta} \right)^{\lambda'-1}  
	\left[ 
	C^{-1}A - C C_\xi 
	\right] 
	- K^{p-1} C  
	\left( r+R(1+At)^\theta \right)^{(\lambda'-1)(p-1)-1} 
\end{aligned}
\]
Since $A>1$, 
we can choose $0<C_\xi<1$ depending only on 
$n$, $p$, $\lambda'$, $R$ and $\theta$ so small that 
$2^{-1} C^{-1}A \geq 2^{-1} C^{-1} \geq C_\xi C$. 
Then by the choice of $\theta>0$ in \eqref{eq:Rtdef}, we see that  
\[
\begin{aligned}
	&\partial_t v_\oute^+ -\Delta_p v_\oute^+ -\xi'(t) \cdot \nabla v_\oute^+ \\
	&\geq 
	\frac{KA}{2C} (1+At)^{\theta-1} 
	\left( r+R(1+At)^{\theta} \right)^{\lambda'-1}  
	- K^{p-1} C  
	\left( r+R(1+At)^\theta \right)^{(\lambda'-1)(p-1)-1} \\
	&=   
	\frac{(1+At)^{\theta-1}}{ ( r+R(1+At)^{\theta} )^{1-\lambda'} }
	\left( 
	\frac{KA}{2C} 
	- K^{p-1} C  (1+At)^{1-\theta} 
	\left( r+R(1+At)^\theta \right)^{-\frac{1}{\theta}}
	\right) \\
	&\geq 
	\frac{(1+At)^{\theta-1}}{ ( r+R(1+At)^{\theta} )^{1-\lambda'} }
	\left( 
	\frac{KA}{2C} 
	- K^{p-1} C R^{-\frac{1}{\theta}} (1+At)^{-\theta} 
	\right) \\
	&\geq 
	\frac{(1+At)^{\theta-1}}{ ( r+R(1+At)^{\theta} )^{1-\lambda'} }
	\left( 
	\frac{KA}{2C} 
	- K^{p-1} C R^{-\frac{1}{\theta}}  
	\right)
\end{aligned}
\]
for $r> \rho(t)$ and $t\geq0$. 
Then there exists $A>1$ depending only 
on $n$, $p$, $\lambda'$, $K$, $R$ and $\theta$ such that 
$v_\oute^+$ is a supersolution for $r> \rho(t)$ and $t\geq 0$.  
By \eqref{eq:Rtdef}, 
we note that $C_\xi$ (resp. $A$) can be determined by 
$n$, $p$, $\lambda$ and $\lambda'$ (resp. $n$, $p$, $\lambda$, $\lambda'$ and $K$). 
Now we fix $C_\xi$. 
On the other hand, we take $A$ large again in the construction of a subsolution.

\subsection{Subsolution}
Let $0<\sigma<1$ be a small constant specified later.
We construct a subsolution of the form
\[
\left\{
\begin{aligned}
	&v_\inn^-(r,t):= K^{-1} (1+At)^{-\frac{1}{p-2}} r^{\lambda'}, \\
	&v_\oute^-(r,t):= K^{-1} \sigma^{\lambda'} (\sigma-\delta)^{-\lambda} 
	(1+At)^{-\frac{1}{p-2}} (r-\delta)^\lambda, 
\end{aligned}
\right.
\]
where $\lambda$ and $\lambda'$ satisfy \eqref{eq:lamcon}, 
$K$ satisfies \eqref{eq:Kcon} and $\delta$ is given by 
\begin{equation}\label{eq:delta_con}
	\delta:= \frac{2\lambda'-\lambda}{2\lambda'} \sigma < \sigma.
\end{equation}
Remark that  $K$ and $A$ will be chosen sufficiently large and 
$C_\xi$ is already fixed in the construction of the supersolution. 
We also remark that $r=\sigma$ is an intersection point 
of $v_\inn^-$ and $v_\oute^-$. Define 
\[
	v^-(r,t):=
	\left\{ 
	\begin{aligned}
	&v_\inn^-(r,t) && \mbox{ for }0\leq r\leq \sigma, \\
	&v_\oute^-(r,t) && \mbox{ for }r> \sigma.
	\end{aligned}
	\right.
\]
Here and below, we write $v^-(y,t)=v^-(r,t)$ by abuse of notation.  
Similarly to the construction of the supersolution, 
we check that $v^-$ becomes a subsolution of \eqref{eq:prad} satisfying 
\begin{equation}\label{eq:v-ord}
	0\leq v^-(r,t) \leq k' r^{\lambda'}  (1+r)^{\lambda-\lambda'}
	\quad 
	\mbox{ for }r\geq0, \; t\geq0. 
\end{equation}

Step 1 (Ordering condition). 
We check \eqref{eq:v-ord}. Remark that $\lambda < \lambda'$ by \eqref{eq:lamcon}. 
For $0\leq r <\sigma\, (<1)$ and $t\geq 0$, we have 
\[
	v^-(r,t)=v_\inn^-(r,t)= K^{-1} (1+At)^{-\frac{1}{p-2}} r^{\lambda'} 
	\leq  K^{-1} r^{\lambda'} 
	\leq k' 2^{\lambda-\lambda'} r^{\lambda'}
	\leq k' (1+r)^{\lambda-\lambda'} r^{\lambda'}
\] 
if $K^{-1}<2^{\lambda -\lambda'} k'$. 
For $\sigma\leq r\leq 1$, by \eqref{eq:delta_con} and $\sigma\leq r$, we see that  
\[
\begin{aligned}
	&v^-(r,t)=v_\oute^-(r,t)= 
	K^{-1} \sigma^{\lambda'-\lambda} 
	\left( \frac{\lambda}{2\lambda'}\right)^{-\lambda} 
	(1+At)^{-\frac{1}{p-2}} (r-\delta)^\lambda \\
	&\leq 
	K^{-1} r^{\lambda'} 
	\left( \frac{\lambda}{2\lambda'}\right)^{-\lambda} 
	\leq
	k'2^{\lambda-\lambda'} r^{\lambda'} 
	\leq 
	k'(1+r)^{\lambda-\lambda'} r^{\lambda'}
\end{aligned}
\]
provided that $K$ satisfies 
$K^{-1} (2\lambda'/\lambda)^\lambda 
\leq k'2^{\lambda-\lambda'}$. 
As for $r > 1$, since $\sigma<1$, we have 
\[
\begin{aligned}
	&v^-(r,t)=v_\oute^-(r,t)\leq 
	K^{-1} \left( \frac{\lambda}{2\lambda'}\right)^{-\lambda} 
	r^\lambda 
	= 
	K^{-1} \left( \frac{\lambda}{2\lambda'}\right)^{-\lambda} 
	r^{\lambda-\lambda'} r^{\lambda'} \\
	&\leq 
	K^{-1} \left( \frac{\lambda}{2\lambda'}\right)^{-\lambda} 
	(1+r)^{\lambda-\lambda'} r^{\lambda'} 
	\leq 
	k' (1+r)^{\lambda-\lambda'} r^{\lambda'} 
\end{aligned}
\]
also provided that $K$ satisfies 
$K^{-1} ( 2\lambda'/\lambda)^\lambda\leq k'$. 
Hence \eqref{eq:v-ord} holds if $K$ is sufficiently large. 
More precisely, taking \eqref{eq:Kcon} into account, 
now we fix $K$ so large that 
\begin{equation}\label{eq:Kcon2}
	K>
	\max\left\{ 
	1, 2^{\lambda'-\lambda} k, \left( 1+\frac{\lambda'}{\lambda} \right)^{\lambda'} k, 
	\frac{2^{\lambda' -\lambda}}{k'}, 
	\left( \frac{2\lambda'}{\lambda} \right)^\lambda \frac{2^{\lambda'-\lambda}}{k'}, 
	\right\}. 
\end{equation}
Recall that the supersolution $v^+$ satisfies \eqref{eq:v+ord}. 
Thus, $v^-$ and $v^+$ also satisfy 
the desired ordering condition \eqref{eq:order}.

Step 2 (Matching condition). 
We need to check the validity of the appropriate matching condition 
We check the matching condition
$\partial_r v_\inn^-(\sigma,t) \leq \partial_r v_\oute^-(\sigma,t)$. 
By \eqref{eq:delta_con}, we have 
\[
\begin{aligned}
	&\partial_r v_\inn^-(\sigma,t) - \partial_r v_\oute^-(\sigma,t) \\
	&= K^{-1}\lambda'(1+At)^{-\frac{1}{p-2}} \sigma^{\lambda'-1} 
	- K^{-1}\lambda \sigma^{\lambda'} 
	(\sigma-\delta)^{-\lambda} (1+At)^{-\frac{1}{p-2}} (\sigma-\delta)^{\lambda-1} \\
	&= K^{-1}(1+At)^{-\frac{1}{p-2}} (\sigma-\delta)^{-1} \sigma^{\lambda'-1}
	\left[\lambda' (\sigma-\delta) - \lambda \sigma \right] \leq 0
\end{aligned}
\]
for $t\geq0$. 
Then the matching condition is satisfied.

Step 3 (Inner part). 
We show that $v_\inn^-$ is a subsolution for $r\leq\sigma$. 
Then, 
\[
\begin{aligned}
	&\partial_t v_\inn^- -\Delta_p v_\inn^- -\xi'(t) \cdot \nabla v_\inn^- \\
	&= -\frac{A}{p-2}K^{-1}(1+At)^{-\frac{1}{p-2}-1} r^{\lambda'} \\ 
	&\quad 
	- K^{-(p-1)}(1+At)^{-\frac{p-1}{p-2}} 
	(\lambda')^{p-1} (n-1-(1-\lambda')(p-1))
	 r^{-1-(1-\lambda')(p-1)} \\
	&\quad 
	-K^{-1}\xi'(t) \cdot (y/|y|) 
	(1+At)^{-\frac{1}{p-2}} \lambda' r^{\lambda'-1}. 
\end{aligned}
\]
By \eqref{eq:lamcon}, we have $(n-1) - (1-\lambda')(p-1)>0$. 
Then by \eqref{eq:assum} and 
$(p-1) - \lambda' (p-2)>0$ due to \eqref{eq:lamcon}, 
there exists $C=C(n,p,\lambda',K)>1$ such that 
\[
\begin{aligned}
	&\partial_t v_\inn^- -\Delta_p v_\inn^- -\xi'(t) \cdot \nabla v_\inn^- \\
	&\leq  
	- K^{-(p-1)}(1+At)^{-\frac{p-1}{p-2}} 
	(\lambda')^{p-1} (n-1-(1-\lambda')(p-1))
	 r^{-1-(1-\lambda')(p-1)} \\
	&\quad 
	-K^{-1}\xi'(t) \cdot (y/|y|) 
	(1+At)^{-\frac{1}{p-2}} \lambda' r^{\lambda'-1} \\
	&\leq 
	- C^{-1} (1+At)^{-\frac{p-1}{p-2}} 
	 r^{-1-(1-\lambda')(p-1)} 
	+ CC_\xi  (1+At)^{-\frac{1}{p-2}-1}  r^{\lambda'-1} \\
	&=  
	- (1+At)^{-\frac{p-1}{p-2}} r^{-1-(1-\lambda')(p-1)} 
	\left( 
	C^{-1} 
	- CC_\xi  r^{(p-1) - \lambda' (p-2)}
	\right) \\
	&\leq 
	- (1+At)^{-\frac{p-1}{p-2}} r^{-1-(1-\lambda')(p-1)} 
	\left( 
	C^{-1} 
	- CC_\xi  \sigma^{(p-1) - \lambda' (p-2)}
	\right)
\end{aligned}
\]
for $r\leq \sigma$. 
We note that $C_\xi$ (resp. $K$) is already determined 
by $n$, $p$, $\lambda$ and $\lambda'$ 
(resp. $k$, $\lambda$ and $\lambda'$). 
Then we now fix $\sigma$ small depending only on 
$n$, $p$, $\lambda'$, $K$ and $C_\xi$ 
such that $v_\inn^-$ is a subsolution for $r\leq\sigma$ and $t\geq0$.

Step 4 (Outer part). 
We show that $v_\oute^-$ is a subsolution for $r>\sigma$. 
For simplicity, we write $b:=\sigma^{\lambda'} (\sigma-\delta)^{-\lambda}$. 
Then, 
\[
\begin{aligned}
	&\partial_t v_\oute^- -\Delta_p v_\oute^- -\xi'(t) \cdot \nabla v_\oute^- \\
	&= -\frac{bA}{p-2}K^{-1} (1+At)^{-\frac{1}{p-2}-1} (r-\delta)^\lambda \\ 
	&\quad 
	+ K^{-(p-1)} b^{p-1} (1+At)^{-\frac{p-1}{p-2}} \lambda^{p-1} 
	(r-\delta)^{(\lambda-1)(p-1)-1} 
	\left[(1-\lambda)(p-1)-(n-1)\frac{r-\delta}{r}\right] \\
	&\quad 
	-K^{-1} \xi'(t) \cdot (y/|y|) 
	b(1+At)^{-\frac{1}{p-2}} \lambda (r-\delta)^{\lambda-1}. 
\end{aligned}
\]
By \eqref{eq:assum} with $C_\xi<1$, 
there exists $C=C(n,p,\lambda,K,b)>1$ such that 
\[
\begin{aligned}
	&\partial_t v_\oute^- -\Delta_p v_\oute^- -\xi'(t) \cdot \nabla v_\oute^- \\
	&\leq 
	- C^{-1} A (1+At)^{-\frac{1}{p-2}-1} (r-\delta)^\lambda 
	+ C (1+At)^{-\frac{p-1}{p-2}} 
	(r-\delta)^{(\lambda-1)(p-1)-1} 
	\left(1+\frac{\delta}{r}\right) \\
	&\quad 
	+ C C_\xi (1+At)^{-\frac{1}{p-2}-1} (r-\delta)^{\lambda-1} \\
	&= 
	- (1+At)^{-\frac{p-1}{p-2}} (r-\delta)^\lambda 
	\left( 
	C^{-1} A 
	- C (r-\delta)^{(p-2)\lambda -p} 
	\left(1+\frac{\delta}{r}\right) 
	- C C_\xi (r-\delta)^{-1}
	\right). 
\end{aligned}
\]
Since $(p-2)\lambda -p<0$, we have 
\[
	(r-\delta)^{(p-2)\lambda -p} 
	\left(1+\frac{\delta}{r}\right)  
	+ C C_\xi (r-\delta)^{-1}
	\leq 
	(\sigma-\delta)^{(p-2)\lambda -p} 
	\left(1+\frac{\delta}{\sigma}\right)  
	+ C C_\xi (\sigma-\delta)^{-1}
\]
for $r>\sigma$. 
Hence there exists $C=C(n,p,\lambda,K,\sigma, \delta, b)>1$ such that 
\[
\begin{aligned}
	\partial_t v_\oute^- -\Delta_p v_\oute^- -\xi'(t) \cdot \nabla v_\oute^- 
	\leq 
	- (1+At)^{-\frac{p-1}{p-2}} (r-\delta)^\lambda 
	\left( 
	C^{-1} A - C  - C C_\xi 
	\right). 
\end{aligned}
\]
We note that $C(n,p,\lambda,K,\sigma, \delta, b)$ can be determined 
by $n$, $p$, $\lambda$, $\lambda'$, $\sigma$ and $K$. 
Hence we can choose $A$ large enough depending only on 
$n$, $p$, $\lambda$, $\lambda'$, $\sigma$ and $K$ 
such that $v_\oute^-$ is a subsolution for $r>\sigma$ 
and $t\geq0$.

\section{Existence}
We now argue for the existence of a solution of \eqref{eq:prad} 
and prove Theorem \ref{th:main}.

\begin{proof}[Proof of Theorem \ref{th:main}]
We divide the proof into 3 steps. 

Step 1 (Solution of the equation). 
We denote by $v^\pm$ the sub- and supersolution of \eqref{eq:prad} 
obtained in the previous section. For $i\geq 1$, we introduce
\[
	\Omega_i = B_{i+1}\setminus B_{\frac{1}{i+1}},
\]
the smooth functions $\eta_i$ by
\[
\left\{
\begin{aligned}
	&\eta_i (y)=1, && y\in \Omega_{i-1},\\
	&\eta_i(y)=0, && y\in \R^n\setminus \Omega_i,\\
	&0\leq \eta_i(y)\leq 1, && y\in \Omega_i\setminus \Omega_{i-1}, 
\end{aligned}
\right. 
\]
and finally 
\[
	v_{0,i}^\pm(y) :=v_0(y)\eta_i(y)+v^\pm (y,0)(1-\eta_i(y))
\]
for $i\geq 2$. We recall that $v^-(\cdot,0) \leq v_0\leq v^+(\cdot,0)$ 
by \eqref{eq:u0as} and \eqref{eq:order}. 
Then we can see the following conditions:
\begin{equation}\label{eq:relations}
\left\{
\begin{aligned}
	&v^\pm_{0,i}=v^\pm(\cdot,0) && \mbox{ in } \R^n\setminus \Omega_i,\\
	&v^\pm_{0,i+1}=v_0 && \mbox{ in } \Omega_i,\\
	&v^-(\cdot,0)\leq v^-_{0,i}\leq v^-_{0,i+1}=v_0 
	=v^+_{0,i+1}\leq v^+_{0,i}\leq v^+(\cdot,0), && \mbox{ in }\Omega_i.
\end{aligned}
\right. 
\end{equation}

We now consider $v_i^\pm=v_i^\pm(y,t)$ to be the solutions of 
\begin{equation}\label{eq:ieq}
\left\{ 
\begin{aligned}
	&\partial_t v_i^\pm= \Delta_p v_i^\pm  + \xi'(t) \cdot \nabla v_i^\pm  
	&&\mbox{ in } \Omega_i\times (0,i),\\
	&v_i^\pm =v^\pm && \mbox{ on }\partial \Omega_i \times (0,i),\\
	&v_i^\pm(\cdot,0) = v^\pm_{0,i} && \mbox{ in }\Omega_i.
\end{aligned}
\right. 
\end{equation}
The existence of a weak solution follows since $\xi$ is locally uniformly $C^1$ 
and since $v^\pm$ and $v_{0,i}^\pm$ are $C^1$ and $\partial \Omega_i$ is smooth, 
see \cite[Theorem 1.2, Section 1.3]{BDM}. 
From \cite[Theorem 1.2, page 42]{DiB}, 
it follows that $v_i^\pm\in C^{\alpha}(\Omega_i\times (0,i))
\cap C(\overline\Omega_i\times [0,i])$ 
for some $\alpha\in (0,1)$. 
We observe that $v^\pm$ are sub- and supersolutions of \eqref{eq:ieq}. 
In addition, due to \eqref{eq:relations}, $v^\pm_{i+1}$ 
are super- and subsolutions of \eqref{eq:ieq}. 
Therefore, the comparison principle implies that 
\[
	v^-\leq v^-_i\leq v^-_{i+1}\leq v^+_{i+1}\leq v^+_i\leq v^+
	\quad  \mbox{ on }\overline \Omega_i\times [0,i], 
\]
where this can be proved by choosing the test function 
$\phi = \max\{ w-\tilde w, 0\}$ 
in the weak formulation given a subsolution $w$ and a supersolution $\tilde w$. 
Hence, for each $j\geq i$, we have 
\begin{equation}\label{eq:jbound}
	v^-\leq v^-_i\leq v^-_j\leq v^+_j 
	\leq v^+_i \leq v^+ 
	\mbox{ on }\overline \Omega_i\times [0,i].
\end{equation}

The idea is now to study
\[
	\lim_{j\to \infty}v^-_j(y,t) \quad \mbox{ for }
	(x,t)\in \R^n\times [0,\infty)
\]
and verify that this is indeed a solution of \eqref{eq:main}. 
From \eqref{eq:jbound} together with interior estimates 
(see \cite[Theorem 1.2, page 42]{DiB}), 
it follows that for each $j\geq i$, there exists $\alpha \in (0,1)$ such that
\[
	\|v^-_{2j}\|_{L^\infty(\Omega_{2i}\times (0,2i)}
	+ \|v^-_{2j}\|_{C^{\alpha}(\Omega_{i}\times (0,i)}\leq C(i).
\]
By a standard diagonalization argument, 
we may extract a subsequence $v_{j'}$ and a limiting function $v$ such that 
$v_{j'}\to v
$ locally uniformly in $C^{\alpha}((\R^n\setminus \{0\})\times (0,\infty)$ 
and weakly in $L^p_{\loc}((0,\infty); W_\text{loc}^{1,p}(\R^n\setminus\{0\}))$. 
It is routine to verify that $v$ is a weak solution of 
$\partial_t v = \Delta_p v  + \xi'(t) \cdot \nabla v$ 
in $(\R^n\setminus\{0\})\times (0,\infty)$.

Step 2 (Continuity at $t=0$). 
From the fact that $v^\pm_{i+1}(\cdot,0)=v_0$ on $\Omega_i$ 
(see \eqref{eq:relations}), 
it follows that $v(\cdot,0) = v_0$ in $\R^n\setminus\{0\}$. 
We now verify that $v$ is continuous at $t=0$. 
Take $y_0\in \R^n\setminus \{0\}$. 
Then we find $i_0$ such that $y_0\in \Omega_{i_0}$. 
Then \eqref{eq:jbound} implies
\[
	v(y_0,t)-v_0(y_0)\leq v^+_{i_0}(y_0,t)-v_0(y_0).
\]
As for each fixed $y_0$ and $i_0$, the function $v^+_{i_0}(y_0,t)$ 
is continuous at $t=0$. 
This implies
\[
	\limsup_{t\to 0}v(y_0,t)-v_0(y_0)\leq 0.
\]
The bound from below follows in a similar way, instead using $v^-_{i_0}$. 
Hence, $t\mapsto v(y,t)$ is continuous outside the origin.

At the origin, the inequalities in \eqref{eq:jbound} imply that 
\[
	v^-\leq v\leq v^+ \quad \mbox{ in }(\R^n\setminus\{0\})\times [0,\infty). 
\]
Since both functions $v^\pm$ tend to zero as $y\to 0$ 
with a locally uniform (in $t$) decay rate dictated 
by $\lambda$ and $\lambda'$ respectively, 
this implies that $v(y,t)$ is locally uniformly (in $t$) continuous at $y=0$. 
The above now implies that $v(y,t)\to v_0(y,0)$ as $t\to 0$ locally uniformly 
in $y\in\R^n$. 
In addition, we can see that $u$ belongs to $C(\R^n\times[0,\infty))$.

Step 3 (Going back to $u$). 
We may now define $u(x,t)=v(x+\xi(t),t)$. 
It is plain to deduce that $u$ is a weak solution of \eqref{eq:main}
and that for each $R >0$, $u$ enjoys the bound
\begin{equation}\label{eq:uplownear}
	C^-(R)|x-\xi(t)|^{\lambda'} \leq u(x,t)\leq C^+(R)|x-\xi(t)|^\lambda
\end{equation}
for $t\in (0,R)$ as $x\to \xi(t)$. 
Moreover, by \cite{W} (see also  \cite{DF}), 
$\nabla u$ exists and is locally H\"older continuous 
in $\{(x,t); x\in \R^n\setminus \{\xi(t)\}, t\in (0,\infty)\}$. 
In particular, the first inequality in \eqref{eq:uplownear} 
and $\lambda'<1$ in \eqref{eq:lamcon} imply that 
\[
	\lim_{x\to \xi(t)}|\nabla u(x,t)|=\infty.
\]
Hence we have proved the desired conditions (i), (ii) and (iii). 
The proof is complete. 
\end{proof}

\section*{Acknowledgments}
The authors wish to express their sincere gratitude 
to Prof.~Marek Fila for leading us to the problem 
concerning the existence of moving gradient singularities 
in the evolutionary $p$-Laplace equation. 
They also wish to thank Dr. Petra Mackov\'a 
for fruitful discussions. 

The first author has been supported by the Swedish Research Council, 
grant no.~2023-03471. The second author has been supported 
by JSPS KAKENHI, grants nos. 22H01131, 22KK0035 and 23K12998.

Part of this material is based upon work supported by the Swedish Research Council under grant no.~2016-06596 while the first author and Marek Fila were participating in the research program ``Geometric Aspects of Nonlinear Partial Differential Equations'', at Institut Mittag-Leffler in Djursholm, Sweden, during the fall of 2022.
\section*{Declaration}
The authors declare that they have no known competing 
financial interests or personal relationships 
that could have appeared to influence the work reported in this paper,
and that the paper is not currently submitted to other journals and that 
it will not be submitted to other journals during the reviewing process.

\end{document}